\documentclass[sn-mathphys]{sn-jnl}

\jyear{2021}%

\theoremstyle{thmstyleone}%
\newtheorem{theorem}{Theorem}[section]
\theoremstyle{thmstyletwo}%
\newtheorem{example}{Example}[section]%
\newtheorem{remark}{Remark}[section]%

\theoremstyle{thmstylethree}%
\newtheorem{definition}{Definition}[section]%
\newtheorem{lemma}{Lemma}[section]
\newtheorem{corollary}{Corollary}[section]
\begin{document}

\title{ A new method for estimating the real roots of real differentiable functions}


\author*[1]{\fnm{Hassan} \sur{Khandani}}\email{ Hassan.Khandani@iau.ac.ir, Khandani.hassan@yahoo.com}

\author[2]{\fnm{Farshid} \sur{Khojasteh}}\email{Fa\_khojasteh@iau.ac.ir}
\equalcont{These authors contributed equally to this work.}

\equalcont{These authors contributed equally to this work.}

\affil*[1]{\orgdiv{Department of mathematics, Mahabad Branch}, \orgname{Islamic Azad university}, \orgaddress{ \city{Mahabad}, \country{Iran}}}

\affil[2]{\orgdiv{Department of Mathematics, Arak-Branch}, \orgname{Islamic Azad University}, \orgaddress{\city{Arak}, \country{Iran}}}

\abstract{We introduce a new type of Krasnoselskii's result. Using a simple differentiability condition, we relax the nonexpansive condition in Krasnoselskii's theorem.  More clearly, we analyze the convergence of the sequence $x_{n+1}=\frac{x_n+g(x_n)}{2}$ based on some differentiability condition of $g$ and present some fixed point results. We introduce some iterative sequences that for any real differentiable function $g$ and any starting point $x_0\in \mathbb [a,b]$  converge monotonically to the nearest root of $g$ in $[a,b]$ that lay to the right or left side of $x_0$.  Based on this approach, we present an efficient and novel method for finding the real roots of real functions. We prove that no root will be missed in our method. It is worth mentioning that our iterative method is free from the derivative evaluation which can be regarded as an advantage of this method in comparison with many other methods. Finally, we illustrate our results with some numerical examples.}

\maketitle
\section{Introduction}\label{sec1}
We start our discussion by  stating the Krasnoselskii's theorem and some generalizations of this result as follows:
\begin{definition}(Krasnoselskii's sequence \cite{berinde2007iterative})\label{krseq}
Let $X$ be a Banach space and $A$ be a non empty subset of $X$. For each $x_0\in A$ and $\lambda\in (0,1)$, the Krasnoseelskii iteration of $f:A\to A$ is defined as follows:
\begin{equation}
x_{n+1}=\lambda x_n+(1-\lambda)f(x_n)\text{ for each } n\ge 0.
\end{equation}
\end{definition}
\begin{theorem}[Krasnoselskii \cite{krssnoselskiitwo}]
 Suppose $A$ is a uniformly convex compact subset of a Banach space $X$. For every nonexpansive mapping $f:A\to A$  (i.e., $\lvert f(x)-f(y)\rvert\le \lvert x-y\rvert$ for each $x,y\in A$) the sequence of iterations defined by Equation \ref{m11} converges to a fixed point of $f$ .
\end{theorem}
Edelstein extended this result to a Banach space $X$ that has a strictly convex norm \cite{edelstein1966remark}. Then, Ishikawa showed that this result is true for an arbitrary Banach space \cite{edelstein1978nonexpansive}. Moreover, Edelstein and O'brien independently proved this result too. Kirk, weakening the nonexpansive condition, extended this result to directionally nonexpansive mappings in a new approach \cite{kirk2000nonexpansive}.\\
Bailey gave proof of Krasnoselskii's result for nonexpansive real functions on a closed interval \cite{subrahmanyam2018elementary}, \cite{bailey1974krasnoselski}. Bruce P. Hillam extended Bailey's result to some Lipshitzian functions \cite{hillam1975generalization}. For some other studies of fixed points of real functions we refer the reader to \cite{borwein1991fixed}, \cite{herceg1996convergence}.\\
In all these and other similar results, the related mapping is assumed to be nonexpansive or Lipshitzian \cite{berinde2007iterative}. Dealing with real functions and using a slight differentiability condition, we relax the nonexpansive or Lipshitz condition. In this way, we introduce a new type of Krasnoselskii's results. First, we present some fixed point results for these types of functions. Then, we present a new iteration method to find the root of real-valued functions. First, we present a brief introduction of root-finding methods as follows.\\
Most of the time, root-finding methods are based upon iterative sequences. More clearly, for a given function an auxiliary function is defined, then an initial guess is made. Calculating the auxiliary function at this point better approximates the initial guess. Continuing this process will define an iterative sequence that may converge to the guessed root. There are many different root-finding methods such as the Bracketing method, Bisection method, False position method, Newton's method or a combination of these methods that converge to a guessed root \cite{epperson2021introduction}. Sometimes, none of these methods can find all the roots. Generally, there should be some different methods at our disposal when we are looking for roots. This manuscript introduces a new method that at least theoretically can find all the real roots.\\

 To begin with our discussion, let   $f:[a,b]: \subset \mathbb R\to \mathbb R$ be an arbitrary function. $x\in [a,b]$ is called a fixed point of $f$ if $f(x)=x$ \cite{granas2003elementary, agarwal2001fixed}. For the function $f$ we define the auxiliary function $F(x)=\frac{x+f(x)}{2}$ for each $x\in [a,b]$. Every fixed point of $F$ is also a fixed point of $f$ and vice versa. Set $\lambda =\frac{1}{2}$ in the Definition \ref{krseq} and define the iteration sequence $\{x_n\}$ as follows:
\begin{equation}\label{m11}
x_{n+1}=\frac{x_n+f(x_n)}{2}\text{ for each } n\ge 0.
\end{equation}
 We study the convergence of the above sequence. A question naturally arises is that: is the sequence $\ref{m11}$ defined above convergent for every $x_0=a\in [0,1]$?  For another counterexample, we refer the reader to \cite{subrahmanyam2018elementary}. In Example \ref{counterexample} we present another counterexample. We find some necessary conditions under which this sequence is convergent and based on this study we find the fixed points of $f$. Furthermore, we take a step forward and find all the real roots of $f$ as well. We show that our method will not miss any real root (see Corollary $\ref{nomiss}$ ).\\
This paper is organized as follows. First, we study a real function $f$ on an interval $[a,c]$ or $[c,b]$, where $c$ is the only fixed point of $f$ in these intervals. Under some differentiability conditions for $f$, we show that the sequence $\ref{m11}$ always converges to the fixed point $c$ for each starting point $x_0$ in $[a,c]$ or $[c,b]$ (lemma $\ref{existextension1}$ and lemma $\ref{existextension2}$). Then, based on this results, we start looking for the real roots of $f$ as well (Corollaries \ref{root1}, \ref{root2}). Then, we make a comparison between our methods and the Newton-Raphson method. Finally, we illustrate our results with some examples.\\
The Newton-Raphson sequence for the function $f$ that will be needed in the sequel is defined as follows.
\begin{equation}\label{newtonseq}
x_{n+1}=x_n-\frac{f(x_n)}{f^{'}(x_n)}\text { for all } n\ge 0.
\end{equation}
Many other iterative methods estimate the roots of real functions, such as the secant method, and the Ridder, Halley, and Muller methods. For a more complete list of these methods we can refer the reader to \cite{suli2003introduction}, \cite{amin}, \cite{oftadeh2010novel}. For the convergence of the Newton-Raphson and most of these methods, the starting point should be chosen very close to the guessed root. However, in our approach, it does not matter how far the starting points are from the guessed root. It is worth mentioning that, theoretically, our method can be used as a tool to study other iterative methods as well. For an analysis of the convergence of the Newton-Raphson sequence via our method, we refer the reader to \cite{hassan}.

In this manuscript, we show the set of real numbers by $\mathbb R$. We denote the set  $\{0,1,2,\dots \}$  of nonnegative integers by $\mathbb N$. For each $a,b\in \mathbb R$ with $a<b$, $[a,b]=\{x\in \mathbb R:a\le x\le b\}$ and $(a,b)=\{x\in \mathbb R:a< x< b\}$ are called closed and open interval from $a$ to $b$ respectively. Let $f$ be a real-valued function on $\mathbb R$  we show the left  and right derivative of f at $a$ by $f^{'}(a-),f^{'}(a+)$ respectively.

\section{main results}\label{sec2}
To begin our study, we present the following two results to show that when $x_0$ is close enough to a fixed point of $f$, the sequence $x_{n+1}=\frac{x_n+f(x_n)}{2}$ converges to $c$.
\begin{lemma}\label{existleft}
Let $\delta > 0$, $ c\in \mathbb R$ and $f$ be a continuous real function on $[c-\delta,c]$ that is left differentiable at $c$ with $f^{'}(c-)\ge -1$ such that $f(x) > x$ for any $x\in(c-\delta,c)$ and $c$ is the unique fixed point of $f$ in $[c-\delta,c]$. Assume $x_0\in (c-\delta,c)$ and define for each $n\ge 0$:
\begin{equation}\label{mean}
x_{n+1}=\frac{x_n+f(x_n)}{2}.
\end{equation}
If $\delta$ is chosen small enough, then the sequence $\{x_n\}$ converges to $c$.
\end{lemma}
\begin{proof}
Define $g(x)=\frac{f(x)-x}{x-c}$ for any $x\in (c-\delta,c)$. By our assumption $g(c-)=\lim_{x\to c-} \frac{f(x)-x}{x-c}=\frac{f^{'}(c-)-1}{1}\ge -2$. So,there is some $\delta > \sigma > 0$ such that $\frac{f(x)-x}{x-c}\ge-2$ for every $x\in (c-\sigma,c)$. Therefore, for each $x\in (c-\sigma,c)$ we have:
\begin{equation}\label{mean}
\frac{f(x)+x}{2}\le c.
\end{equation}
If $x_0\in (c-\sigma,c)$, then $x_1=\frac{f(x_0)+x_0}{2} > x_0$ and $x_1=\frac{f(x_0)+x_0}{2}\le c$. Suppose $n\ge 1$ be a an integer such that $c-\sigma \le x_{n-1}\le x_n\le c$. By assumption $f(x_n)>x_n$ which follows that $x_{n+1}=\frac{f(x_n)+x_n}{2}>x_n$. By \ref{mean} we have $x_{n+1}=\frac{f(x_n)+x_n}{2}\le c$
Therefore, by induction we see that $x_n\le c$ and $x_{n+1}\ge x_n$ for every $n\ge 0$. Hence, $\{x_n\}$ converges to some $a\in [c-\sigma,c]$. It is easy to see that $f(a)=a$, so by the uniqueness assumption $a=c$ and the proof is complete.
\end{proof}

\begin{lemma}\label{exisright}
Let $\delta>0$, $ c\in R$ and $f$ be a continuous real-valued function on $[c,c+\delta]$ that is right differentiable at $c$ with $f^{'}(c+)\ge -1$ such that $f(x)<x$ for each $x\in(c,c+\delta)$ and $c$ is the unique fixed point of $f$ in $[c,c+\delta]$. Let $c\not =x_0\in (c,c+\delta)$ and for each $n\ge 0$ define:
\begin{equation}
x_{n+1}=\frac{x_n+f(x_n)}{2}.
\end{equation}
If  $\delta$ is chosen small enough, then the sequence $\{x_n\}$  converges to $c$.
\end{lemma}
\begin{proof}
Define $g(x)=\frac{f(x)-x}{x-c}$ for each $x\in (c,c+\delta)$. By our assumption $g(c+)\ge -2$, so there exists some $\delta>\sigma>0$ such that $\frac{f(x)-x}{x-c}\ge-2$ for each  $x\in (c,c+\sigma)$. Therefore, for each $x\in (c,c+\sigma)$ we have:
\begin{equation}
\frac{f(x)+x}{2}\ge c.
\end{equation}
If $x_0\in (c,c+\sigma)$, then $x_1=\frac{f(x_0)+x_0}{2}<x_0$ and $x_1=\frac{f(x_0)+x_0}{2}\ge c$. Arguing by induction we deduce that $x_n\ge c$ and $x_{n+1}\le x_n$ for each $n\ge 0$. Therefore, $\{x_n\}$ converges to some $a\in [c,c+\sigma]$. It is easy to see that $f(a)=a$. So by uniqueness assumption $a=c$ and the proof is complete.
\end{proof}
 To ensure the convergence of the sequence $x_{n+1} =x_n+\frac{f(x_n)}{2}$ in the neighborhood of $c$ in the preceding two lemmas, we only needed the one-sided derivative of $f$ at $c$. However, when we begin looking for a fixed point $c$ of the function $f$ on some interval, we have no idea whether the condition $f^{'}(c-)\ge-1$ or $f^{'}(c+)\ge-1$  holds or not. Therefore, strengthen the differentiability condition we make it practically more useful as follows.
\begin{lemma}\label{existextension1}
Let $f$ be a continuous real-valued function on $[a,c]$ that is differentiable on $(a,c)$  with $f^{'}(x)\ge -1$ on $(a,c)$, $f(x)>x$ for each $x\in[a,c)$, and $c$ is the unique fixed point of $f$ in $[a,c]$. Let $x_0\in [a,c)$ and for each $n\ge 0$ define:
\begin{equation}\label{meanex}
x_{n+1}=\frac{x_n+f(x_n)}{2},
\end{equation}
then the sequence $\{x_n\}$  converges to $c$.
\end{lemma}
\begin{proof} Assume there exists $x\in [a,c)$ such that $\frac{f(x)-x}{x-c}<-2$. Define $h(x)=f(x)-x$ for each $x\in [a,c]$  and applying Lagrange mean value theorem for $h$ \citep{rudin1976principles}, there exists $\sigma \in (c,x)$ such that:
\begin{equation}
-2\le h^{'}(\sigma)=\frac{h(x)-h(c)}{x-c}=\frac{f(x)-x}{x-c}<-2,
\end{equation}
that is a contradiction. Therefore, for each $x\in [a,c)$ we have $\frac{f(x)-x}{x-c}\ge-2$. Equivalently, we have:
\begin{equation}
\frac{f(x)+x}{2}\le c\text { for each } x\in [a,c).
\end{equation}
Now, define $\{x_n\}$ as in equation $\ref{meanex}$. For each $n\ge 0$, $x_n\le c$ and $x_{n+1}\ge x_n$. Therefore, $\{x_n\}$ converges to some point $b\in[a,c]$ with $f(b)=b$. Since $c$ is the unique fixed point of $f$  in $[a,c]$ we deduce that $b=c$ that completes the proof.
\end{proof}
\begin{remark}
 Example $\ref{counterexample}$ shows that the differentiability condition, $f^{'}(x)\ge -1$ on $(a,c)$, in Lemma $\ref{existextension1}$ is necessary.
\end{remark}
\begin{lemma}\label{existextension2}
Let $f$ be a continuous real-valued function on $[c,b]$ that is differentiable on $(c,b)$  with $f^{'}(x)\ge -1$ on $(c,b)$, $f(x)<x$ for each $x\in(c,b]$, and $c$ is the unique fixed point of $f$ in $[c,b]$. Let $x_0\in (c,b]$ and for each $n\ge 0$ define:
\begin{equation}\label{meanex2}
x_{n+1}=\frac{x_n+f(x_n)}{2},
\end{equation}
then the sequence $\{x_n\}$  converges to $c$.
\end{lemma}
\begin{proof}
 Similar to the proof of Lemma $\ref{existextension1}$, for each $x\in [a,c)$ we have $\frac{f(x)-x}{x-c}\ge-2$. Since $f(x)<x$ and $x>c$ for each $x\in (c,b]$, we deduce that  $\frac{x-f(x)}{2}\le x-c$. So we have:
\begin{equation}
\frac{f(x)+x}{2}\ge c\text { for each } x\in [c,b).
\end{equation}
Now, define $\{x_n\}$ as in equation $\ref{meanex2}$. For each $n\ge 0$, $x_n\ge c$ and $x_{n+1}\le x_n$. Therefore, $\{x_n\}$ converges to some point $d\in[c,b)$ with $f(d)=d$. Since $c$ is the unique fixed point of $f$  in $[c,b]$, we deduce that $d=c$ and the proof is complete.
\end{proof}
\begin{remark}
By Example $\ref{notbereplaced}$, in Lemma $\ref{existextension1}$ the assumption  $f(x)>x$ for each $x\in[a,c)$ can not be replaced with  $f(x)<x$ for each $x\in[a,c)$. Also, in Lemma $\ref{existextension2}$ the assumption  $f(x)<x$ for each $x\in(c,b]$ can not be replaced with  $f(x)>x$  for each $x\in(c,b]$.
\end{remark}
Now, we look for the roots of a given function instead of its fixed points.
\begin{corollary}\label{root1}
Let $f$ be a continuous real-valued function on $[a,b]$  and $c\in (a,b)$  be the unique root of $f$ in $[a,b]$. Suppose that $f$   is differentiable on $A=(a,c)\cup(c,b)$ with $f^{'}(x)\ge -2$ for each  $x\in  A$. Also, suppose that $f(x)>0$ for each $x\in[a,c)$ and $f(x)<0$ for each $x\in(c,b]$. For each $x_0\in [a,b]$ and for each $n\ge 0$ define:
\begin{equation}\label{meanexco1}
x_{n+1}=x_n+\frac{f(x_n)}{2},
\end{equation}
then the sequence $\{x_n\}$  converges to $c$.
\end{corollary}
\begin{proof}Define $F(x)=f(x)+x$ for each $x\in [a,b]$. $F(x)>x$ for each $x\in [a,c)$ and  $F$ satisfies all conditions of  Lemma $\ref{existextension1}$ on $[a,c]$. Let $x_0\in[a,c]$. Define $\{x_n\}$  by:
\begin{equation}
x_{n+1}=\frac{F(x_n)+x_n}{2}=x_n+\frac{f(x_n)}{2}\text{ for all } n\ge 0,
\end{equation}
 $\{x_n\}$ converges to $c$  by Lemma $\ref{existextension1}$. $x_0\in [c,b]$ and $F$ satisfies all conditions of  Lemma $\ref{existextension2}$ on $[c,b]$. Now, the above sequence converges to $c$. So the proof is complete.
\end{proof}
\begin{remark}\label{notconverges}
 In Corollary $\ref{root1}$ suppose $f(x)<0$ for each $x\in[a,c)$ and $f(x)>0$ for each $x\in(c,b]$. Then, the sequence $\ref{meanexco1}$ does not converge to $c$ for any $x_0\in A=[a,c)\cup (c,b]$. To see this, assume that $x_0\in[a,c)$. $f(x_0)<0$, therefore, $x_{1}=x_0+\frac{f(x_0)}{2}<x_0$. By induction we see that $x_{n+1}<x_{n}<\dots <x_0<c$. Therefore, $x_n\not \to c$ as $n\to \infty$. The same is true when $x_0\in (c,b]$. We define the sequence $\ref{meanexco2}$ in the following Corollary to converges to these kind of roots.
\end{remark}
\begin{corollary}\label{root2}
Let $f$ be a continuous real-valued function on $[a,b]$  and $c\in (a,b)$  be the unique root of $f$ in $[a,b]$. Suppose that $f$   is differentiable on $A=(a,c)\cup(c,b)$ with $f^{'}(x)\le 2$ for each  $x\in  A$. Also, suppose that $f(x)<0$ for each $x\in[a,c)$ and $f(x)>0$ for each $x\in(c,b]$. For each $x_0\in [a,b]$ and for each $n\ge 0$ define:
\begin{equation}\label{meanexco2}
x_{n+1}=x_n-\frac{f(x_n)}{2},
\end{equation}
then the sequence $\{x_n\}$  converges to $c$.
\end{corollary}
\begin{proof}Define $F(x)=-f(x)+x$ for each $x\in [a,b]$. $F(x)>x$ for each $x\in [a,c)$ and  $F$ satisfies all conditions of  Lemma $\ref{existextension1}$ on $[a,c]$. Let $x_0\in[a,c]$ and define $\{x_n\}$  by:
\begin{equation}
x_{n+1}=\frac{F(x_n)+x_n}{2}=x_n-\frac{f(x_n)}{2}\text{ for all } n\ge 0,
\end{equation}
 $\{x_n\}$ converges to $c$  by Lemma $\ref{existextension1}$. If $x_0\in [c,b]$, then $F$ satisfies all conditions of  Lemma $\ref{existextension2}$ on $[c,b]$ and again the above sequence converges to $c$. So the proof is complete.
\end{proof}
\begin{remark}\label{plusminus}
To be able to refer to them easily, we denote the sequence $\{x_n\}$ in $\ref{meanexco1}$ and $\ref{meanexco2}$ by $\{x^{+}_n\}$ and $\{x^{-}_n\}$ respectively.
\end{remark}

Now we are ready to provide a practical result to look for the roots of a given function on a given interval that satisfies some conditions as follows.
\begin{corollary}\label{nomiss}
Let  $f:\mathbb R\to \mathbb R$ be a real function, $x_0\in \mathbb R$ and $c$ be the nearest root of $f$ to $x_0$ such that $x_0<c$ and $\lvert  f^{'}(x)\rvert \le 2$ for each $x\in [x_0,c]$. Then, one of the following sequences converges $c$. This is also true when $c$ is the nearest root of $f$ to $x_0$ such that $c<x_0.$
\begin{flalign*}
& \hspace{.5cm} x^{+}_{n+1}=x^{+}_n+\frac{f(x^{+}_n)}{2},\\
& \hspace{.5cm}x^{-}_{n+1}=x^{-}_n-\frac{f(x^{-}_n)}{2}.
\end{flalign*}
\end{corollary}
\begin{proof} First, suppose $f(x_0)>0$ and $c$ be the nearest root of $f$ to $x_0$ such that $x_0<c$. $f$ has no root in $[x_0,c)$, so $f(x)>0$ for each $x\in [x_0,c)$. Now, $f$ satisfies all conditions of Corollary $\ref{root1}$. Therefore, the sequence $\{x^{+}_n\}$ converges to $c$. Now, suppose $f(x_0)<0$. Using Corollary $\ref{root2}$, similarly we deduce that  $\{x^{-}_n\}$ converges to $c$. The proof is similar when $c$ is the nearest root of $f$ to $x_0$ and $c<x_0.$ \\

\end{proof}
\begin{theorem}\label{practical}
Let $f$ be a continuous real-valued function on $[a,b]$  that is differentiable on $(a,b)$. Then,
\begin{itemize}
\item[(i)] if $f^{'}(x)\ge -2$ for each $x\in (a,b)$, $f(a)>0,f(b)<0$  and  $f$ has a unique root $c$  in $(a,b)$. Then, the sequence $\ref{plus}$ converges to $c$ for each $x_0\in [a,b]$.
\begin{equation}\label{plus}
x_{n+1}=x_n+\frac{f(x_n)}{2} \text { for each }n\ge 0,
\end{equation}
\item[(ii)] if $f^{'}(x)\le 2$ for each $x\in (a,b)$, $f(a)<0,f(b)>0$  and  $f$ has a unique root $c$  in $(a,b)$. Then, the sequence $\ref{minus}$ converges to $c$ for each $x_0\in [a,b]$.
\begin{equation}\label{minus}
x_{n+1}=x_n-\frac{f(x_n)}{2} \text { for each }n\ge 0.
\end{equation}
\end{itemize}

\end{theorem}
\begin{proof}
First, suppose $f$ satisfies in $(i)$. Let  $x<c$. If $f(x)<0$, since $f(x)f(a)<0$, $f$ has a root in $(a,x)$ that contradicts with the uniqueness of $c$. Therefore, $f(x)>0$ for each $x\in [a,c)$. Similarly, $f(x)<0$ for each $x\in (c,b]$. Now, $f$ satisfies all conditions of Corollary $\ref{root1}$ on the interval $[a,b]$, therefore the sequence $\ref{plus}$ converges to $c$ for each $x_0\in [a,b]$\\
Now, suppose $f$ satisfies in $(ii)$. We see that $f(x)<0$ for each $x\in [a,c)$ and $f(x)>0$ for each $x\in (c,b]$. Consider that $f$ satisfies all conditions of  Corollary $\ref{root2}$. So, the sequence $\ref{minus}$ converges to $c$. If $x_0=c$ the proof is evident.\\
\end{proof}
\begin{remark}
It is worth to mention that in Theorem \ref{practical} we are not allowed to replace the Equation $\ref{minus}$ with  Equation $\ref{plus}$. See Example $\ref{notallowed}$ for a counterexample in this regard.
\end{remark}
\section{Examples}
In this section, we illustrate our results with some examples. \\
The following example shows that for a continuous function $f:[a,b]\to [a,b]$ the sequence $x_{n+1}=\frac{x_n+f(x_n)}{2}$ does not converge to a fixed point of $f$ for some $x_0\in [a,b]$.
\begin{example}\label{counterexample}Define $f:[0,1]\to[0,1]$ as follows:
\begin{equation*}
f(x)=\begin{cases}
       1 \quad &\text{if} \, 0\le x\le \frac{3}{8} \\
           -2x+\frac{7}{4} \quad &\text{if} \, \frac{3}{8}\le x\le \frac{7}{12} \\
          -6x+\frac{49}{12}  \quad &\text{if} \, \frac{7}{12}\le x\le \frac{49}{72} \\
           0 \quad &\text{if} \, \frac{49}{72}\le x\le  1 \\
     \end{cases}
\end{equation*}
$f$ is continuous and $x=\frac{7}{12}$ is the unique fixed point of $f$. Let $x_0=\frac{23}{63}$ and for each $n\ge 1$ define:
\begin{equation}
x_{n+1}=\frac{x_n+f(x_n)}{2}.
\end{equation}
We see that $x_1=\frac{86}{126},x_2=\frac{43}{126},x_3=\frac{169}{252},x_4=\frac{23}{63}=x_0$. Therefore, $\{x_n\}$ is not convergent. The reason is that, altough $f$ satisfies all conditions of Lemma $\ref{existextension1}$ on the interval $[\frac{23}{63},\frac{7}{12}]$, but $f$ does not satisfy the differentiability condition at the point $\frac{3}{8}$ of this interval. More clearly, $c=\frac{7}{12}$ is the unique fixed point of $f$ in $[0,c]$, and $f(x)>x$ for each $x\in [0,c)$. But, $f$ is not differentiable  at $x=\frac{3}{8}$ and $f^{'}\not \ge -1$ on $(\frac{3}{8}, c)$.
\end{example}
The following example shows that, in Lemma \ref{existextension1} it is necessary to have $f(x)>x$ on $[a,c)$.
\begin{example}\label{notbereplaced}
For each $x\in [2/3,1]$ define:
\begin{equation*}
f(x)=2x-1.
\end{equation*}
$f$ is a continuous function on $[2/3,1]$ and  $x=1$ is the unique fixed point of $f$ in $[2/3,1]$. $f(x)<x$ for each $x<1$. For each $x_0$  in $[2/3,1)$ we have $f(x_0)<x_0$. Therefore, $x_1=\frac{f(x_0)+x_0}{2}<x_0$. Through induction we see that $x_{n+1}\le x_n<1$ for all $n\ge 0$. Therefore, $\{x_n\}$ does not converges to $1$. So  in Lemma $\ref{existextension1}$ the assumption  $f(x)>x$ for each $x\in[a,c)$ can not be replaced with  $f(x)<x$ for each $x\in[a,c)$. Also suppose that $f(x)=2x-1$ for each $x\in[1,2]$. $c=1$ is the unique fixed point of $f$ in $[1,2]$ and $f(x)>x$ for each $x\in (1,2]$. Let $x_0\in (1,2]$, similarly we see that   $x_{n+1}>x_n>\dots>x_0>1$ for each $n\ge 0$. Therefore, $\{x_n\}$ does not converges to $1$. So in Lemma \ref{existextension2}, the condition $f(x)<x$ on $(c,b]$ can not be replaced with $f(x)>x$ on $(c,b]$.
\end{example}
\begin{example}\label{notallowed}
Let $f:[-1,1]\to[-1,1]$  be defined by:
\begin{equation*}
f(x)=\begin{cases}
      \frac{2}{3} (x-\frac{1}{2})\quad &\text{if} \, -1\le x\le \frac{1}{2} \\
           2x-1\quad &\text{if} \, \frac{1}{2}\le x\le 1 \\
     \end{cases}
\end{equation*}
Let $g:[-1,1]\to[-1,1]$  be defined by:
\begin{equation*}
g(x)=\begin{cases}
      \frac{-2}{3} (x-\frac{1}{2})\quad &\text{if} \, -1\le x\le \frac{1}{2} \\
           -2x+1\quad &\text{if} \, \frac{1}{2}\le x\le 1 \\
     \end{cases}
\end{equation*}
$x=\frac{1}{2}$ is the unique root of $f$ in $[-1,1]$. Let $\{x_n\}$ be the sequence defined  in the Equation \ref{plus} in Theorem $\ref{practical}$ and $x_0=-1$. Then, $x_1=x_0+\frac{f(x_0)}{2}=\frac{-3}{2}\not \in [-1,1]$. The same is true if $x_0=1$. So Equation \ref{minus} can not be replaced with  Equation \ref{plus}  in Theorem $\ref{practical}$. Similarly using $g$ shows that  the Equation \ref{plus} can not be replaced with  Equation \ref{minus} in Theorem $\ref{practical}$.
\end{example}
In the following example, we make a  comparison between Newton-Raphson and our method. We choose some different starting points and verify how our method converge to the related root. We denote the Newton-Raphson sequence by $\{x^{N}_n\}$ and our sequences by $\{x^{-}_n\}$, $\{x^{+}_n\}$ as we declared in Remark $\ref{plusminus}$.
\begin{example}\label{newton}
Suppose $f(x)=x^3-2x+2$ for each $x\in \mathbb R$. $f$ has a unique root $c\in (-2,0)$. For $x_0=0$ we make a comparison between Newton method and our method. For Newton method we have: $x^{N}_0=0, x^{N}_1=1,x^{N}_2=0,\dots$. Therefore, we see that the Newton sequence $\{x^{N}_n\}$ oscillates between $0$ and $1$. The sequence $\{x^{-}_n\}$ also does not converges if $x^{-}_0=0$. The reason is that, the condition $f^{'}(x)=3x^2-2\le 2$  on the interval $(c,0]$ does not hold. To fix this problem, we define the function $g$ by $g=\frac{f}{5}$. For each $x\in (c,0)$ we have:
\begin{equation}
g^{'}(x)\le \sup\{ \frac{f^{'}(x)}{5}:x\in (c,0)\}=\sup \{\frac{3x^2-2}{5}:x\in (-2,0)\}\le2.
\end{equation}
Therefore we replace $f$ with  $g=\frac{f}{5}$. Now $f$ satisfies all conditions of Corollary $\ref{root2}$ on $(c,0)$ and the sequence $\{x^{-}_n\}$ with initial point $x_0=0$ converges to $c$.\\
In each interval we divide $f$ by a suitable constant number to apply our method on that interval. $g=f/4$  satisfies all conditions of  Corollary $\ref{root2}$ on $[-2,c]$. Therefore,  $\{x^{-}_n\}$ with initial point $x_0=-2$ converges to $c$. $f^{'}/13\le 2$ on $[-2,3]$ so for the function $f/13$ the sequence $\{x^{-}_n\}$ converges to $c$ for each $x_0\in [c,3]$.
\end{example}
We give the following example to assert that when our method does not work.
\begin{example}Let $f(x)=x^{1/3}$. We know that $x=0$ is the unique root of $f$ in $(-1,1)$. $f(x)<0$ for each $x\in (-1,0)$ and $f(x)>0$ for each $x\in (0,1)$. By Corollary $\ref{root2}$ we should use the sequence $\{x^{-}_n\}$ to find the root. Examining the sequence  $\{x^{-}_n\}$ for some $x_0\in (-1,1)$, we see that the sequence is not convergent. The reason is that $\lim_{x\to 0}f^{'}(x)=+\infty$ and whatever $\delta>0$ we choose $f^{'}$ remains unbounded on $(-\delta,\delta)$. Therefore, the related differentiability condition does not hold and consequently our method can not be applied.
\end{example}
\bmhead{Acknowledgments}
This research did not receive any specific grant from funding agencies in the public, commercial, or not-for-profit sectors.

\bibliography{sn-bibliography}

\end{document}